\documentclass[conference]{IEEEtran}
\IEEEoverridecommandlockouts
\usepackage{cite}
\usepackage{amsmath,amssymb,amsthm}
\usepackage[utf8]{inputenc}
\usepackage{graphicx}
\usepackage{booktabs}
\usepackage{textcomp}
\usepackage{xcolor}
\usepackage{fancyhdr}

\def\BibTeX{{\rm B\kern-.05em{\sc i\kern-.025em b}\kern-.08em
    T\kern-.1667em\lower.7ex\hbox{E}\kern-.125emX}}

\newtheorem{theorem}{Theorem}
\newtheorem{remark}{Remark}
\newtheorem{proposition}{Proposition}

\begin{document}

\title{An Adaptive KKT-Based Indicator for Convergence Assessment in Multi-Objective Optimization}
\author{\IEEEauthorblockN{Thiago Santos}
\IEEEauthorblockA{\textit{Department of Mathematics} \\
\textit{Federal University of Ouro Preto, UFOP}\\
Ouro Preto, Brazil \\
santostf@ufop.edu.br}
\and
\IEEEauthorblockN{Sebastião Xavier}
\IEEEauthorblockA{\textit{Department of Mathematics} \\
	\textit{Federal University of Ouro Preto, UFOP}\\
	Ouro Preto, Brazil \\
	semarx@ufop.edu.br}}

\maketitle
\pagestyle{fancy}
\fancyhf{}
\fancyfoot[LE, LO]{}
\begin{abstract}
	Performance indicators are essential tools for assessing the convergence behavior of multi-objective optimization algorithms, particularly when the true Pareto front is
	unknown or difficult to approximate. Classical reference-based metrics such as
	hypervolume and inverted generational distance are widely used, but may suffer from
	scalability limitations and sensitivity to parameter choices in many-objective scenarios.
	Indicators derived from Karush--Kuhn--Tucker (KKT) optimality conditions provide an
	intrinsic alternative by quantifying stationarity without relying on external reference
	sets. This paper revisits an entropy-inspired KKT-based convergence indicator and proposes a
	robust adaptive reformulation based on quantile normalization. The proposed indicator
	preserves the stationarity-based interpretation of the original formulation while
	improving robustness to heterogeneous distributions of stationarity residuals, a
	recurring issue in many-objective optimization.
\end{abstract}


\section{Introduction}

Performance assessment is a fundamental component of empirical research in multi-objective optimization, as it underpins claims regarding convergence behavior, solution quality, and algorithmic superiority. In recent years, it has become increasingly evident that performance evaluation is not a neutral post-processing step, but rather a methodological choice that can substantially influence experimental conclusions. Comprehensive analyses have shown that different indicators may induce conflicting rankings even when applied to the same approximation sets, an effect that becomes more pronounced as the number of objectives increases \cite{Audet2021}. Consequently, the reliability and interpretability of performance indicators remain central concerns in contemporary many-objective optimization research.

Beyond benchmarking and ranking purposes, convergence indicators also play an important role in computational optimization as algorithmic diagnostics. In particular, they can support the design of stopping criteria, the analysis of transient behavior, and the monitoring of progress when the Pareto front is unknown or impractical to approximate. In this context, indicators that rely on intrinsic optimality information are especially attractive, since they can be applied consistently across problems and experimental settings without requiring external reference sets.

A major source of difficulty in many-objective settings lies in the progressive loss of selection pressure induced by Pareto dominance. As the dimensionality of the objective space grows, dominance relations become increasingly sparse, leading to approximation sets in which most solutions are mutually nondominated. This phenomenon not only complicates algorithm design, but also undermines the effectiveness of dominance-based and dominance-derived performance indicators. As a result, the evaluation of convergence and progress toward Pareto optimality requires alternative mechanisms that remain discriminative under high-dimensional trade-offs.

Among reference-based indicators, the hypervolume measure has been widely adopted due to its strict Pareto compliance and its ability to capture convergence and diversity simultaneously. These properties have made hypervolume a de facto standard in many benchmarking studies. Nevertheless, its practical applicability is constrained by two well-known issues. First, the computational cost of exact hypervolume calculation increases rapidly with the number of objectives, often necessitating approximations. Second, the indicator exhibits a strong dependence on the choice of reference point, which may significantly affect both absolute values and relative comparisons between algorithms \cite{Guerreiro2021}. These limitations have motivated sustained debate regarding the role of hypervolume in many-objective performance assessment.

Distance-based indicators, such as the inverted generational distance (IGD), provide an alternative by quantifying the average proximity of an approximation set to a predefined reference representation of the Pareto front. While conceptually straightforward, IGD inherits a critical dependency on the construction of the reference set. It has been demonstrated that the density, distribution, and placement of reference points can dramatically alter evaluation outcomes, potentially favoring certain solution characteristics while penalizing others \cite{Ishibuchi2016}. In many-objective problems, where accurate and uniformly distributed reference fronts are difficult to obtain, this sensitivity raises concerns about the robustness and reproducibility of IGD-based comparisons \cite{Ishibuchi2018}.

In response to these challenges, a complementary line of research has focused on intrinsic convergence indicators that do not rely on external reference information. For smooth multi-objective optimization problems, first-order necessary conditions for Pareto optimality provide a natural and theoretically grounded basis for convergence assessment. These conditions extend the classical Karush--Kuhn--Tucker (KKT) framework to vector-valued objectives, leading to stationarity concepts that characterize local Pareto optimality. KKT-based proximity measures quantify the extent to which candidate solutions violate these optimality conditions, thereby offering a reference-free indicator directly linked to the problem’s analytical structure \cite{Abouhawwash2016}.

Within this framework, Santos and Xavier proposed an entropy-inspired KKT-based convergence indicator that aggregates stationarity residuals into a bounded scalar measure \cite{Santos2018}. The boundedness of the indicator facilitates interpretation and comparison across runs, while its formulation avoids explicit dependence on reference sets or dominance relations. However, the indicator relies on a fixed saturation parameter that implicitly assumes a uniform scale of stationarity residuals across the approximation set.

In practice, approximation sets often display substantial heterogeneity in convergence status, particularly during intermediate stages of the optimization process. It is common for a population to contain a mixture of nearly stationary solutions and others that remain far from satisfying optimality conditions. Under such circumstances, fixed saturation mechanisms may map a wide range of residual magnitudes to similar indicator contributions, effectively reducing resolution among poorly converged solutions and limiting the indicator’s ability to reflect gradual improvements \cite{Zhang2022}. This loss of discrimination is especially problematic in many-objective scenarios, where convergence progresses unevenly across different regions of the objective space.

Recent contributions to performance assessment have therefore emphasized the need for robustness with respect to heterogeneous solution quality, advocating normalization and scaling strategies that adapt to the statistical structure of approximation sets rather than imposing fixed global parameters \cite{Li2023}. Such approaches aim to enhance consistency and interpretability without altering the underlying convergence concept. Motivated by these considerations, this paper proposes an adaptive reformulation of the entropy-inspired KKT-based convergence indicator, designed to preserve its theoretical grounding while improving discrimination across diverse stationarity regimes.

\section{Background}

We consider a smooth multi-objective optimization problem of the form
\begin{equation}
	\min_{x \in \Omega} \; F(x) = \bigl(f_1(x), f_2(x), \ldots, f_m(x)\bigr),
\end{equation}
where $\Omega \subset \mathbb{R}^n$ is a nonempty feasible set and each objective function
$f_i : \Omega \rightarrow \mathbb{R}$ is continuously differentiable.
A point $x^\star \in \Omega$ is said to be Pareto optimal if there exists no feasible point
$x \in \Omega$ such that $f_i(x) \le f_i(x^\star)$ for all $i \in \{1,\ldots,m\}$ and
$f_j(x) < f_j(x^\star)$ for at least one index $j$.

For smooth problems, Pareto optimality admits a characterization in terms of first-order
necessary conditions.
A feasible point $x^\star$ is Pareto stationary if there exists a vector
$\lambda \in \mathbb{R}^m$ satisfying
\begin{equation}
	\sum_{i=1}^{m} \lambda_i \nabla f_i(x^\star) = 0,
	\qquad
	\lambda_i \ge 0,
	\qquad
	\sum_{i=1}^{m} \lambda_i = 1.
	\label{eq:kkt}
\end{equation}
Condition~\eqref{eq:kkt} generalizes the classical Karush--Kuhn--Tucker framework to the
multi-objective setting and provides a reference-free notion of first-order stationarity
that is widely adopted in continuous multi-objective optimization
\cite{Abouhawwash2016}.

Given a candidate solution $x \in \Omega$, violation of the stationarity condition
\eqref{eq:kkt} can be quantified by solving the convex quadratic program
\begin{equation}
	\min_{\lambda \in \mathbb{R}^m}
	\left\|
	\sum_{i=1}^{m} \lambda_i \nabla f_i(x)
	\right\|^2
	\quad
	\text{subject to}
	\quad
	\lambda_i \ge 0,
	\quad
	\sum_{i=1}^{m} \lambda_i = 1,
	\label{eq:qpp}
\end{equation}
which seeks a convex combination of objective gradients with minimal Euclidean norm.
Let $\lambda^\star(x)$ denote an optimal solution of \eqref{eq:qpp}. The associated vector
\begin{equation}
	q(x) = \sum_{i=1}^{m} \lambda_i^\star(x) \nabla f_i(x)
\end{equation}
vanishes if and only if $x$ satisfies the first-order Pareto stationarity condition.

A scalar measure of stationarity violation can then be defined as
\begin{equation}
	s(x) = \| q(x) \|^2,
	\label{eq:residual}
\end{equation}
which is nonnegative, continuous, and invariant under positive scaling of the objective
functions.
Such residuals form the basis of several KKT-based proximity measures and convergence
indicators, as they provide an intrinsic assessment of optimality that does not rely on
external reference sets or dominance relations
\cite{DebAbouhawwashSeada2017,Eichfelder2021}.

Given a finite approximation set $X = \{x_1, \ldots, x_N\}$, the challenge lies in
aggregating individual residuals $s(x_i)$ into a single scalar quantity that reflects
the overall convergence status of the set.
The entropy-inspired indicator proposed in \cite{Santos2018} addresses this issue by
bounding individual residual contributions through a fixed saturation constant, thereby
ensuring a normalized and interpretable aggregation.
However, this construction implicitly assumes that stationarity residuals exhibit
comparable scales across the approximation set, an assumption that may be violated in
practice when residual distributions are highly heterogeneous or heavy-tailed.

\section{Adaptive $\mathcal{H}$ Indicator}

Let $X = \{x_1,\ldots,x_N\}$ be a finite approximation set.
For each $x_i \in X$, we compute the stationarity residual $s_i = s(x_i)$ defined in \eqref{eq:residual}.
The original entropy-inspired indicator proposed in \cite{Santos2018} aggregates these residuals through a fixed saturation mapping,
\begin{equation}
	\tilde{s}_i = \min \left\{ \frac{1}{e}, \; s_i \right\},
\end{equation}
followed by the entropy-like aggregation
\begin{equation}
	\mathcal{H}_{old}(X)
	=
	-\frac{1}{N}
	\sum_{i=1}^{N}
	\tilde{s}_i \log \bigl( \tilde{s}_i \bigr),
	\label{eq:Hold}
\end{equation}
with the convention $0 \log 0 = 0$.

The fixed saturation threshold employed in $\mathcal{H}_{old}$ implicitly assumes that stationarity residuals are well scaled and lie within a bounded range.
In heterogeneous approximation sets, however, residuals may span several orders of magnitude.
In such cases, large residual values are mapped to identical contributions, which reduces the indicator’s discriminatory power and limits its ability to reflect gradual convergence improvements.
To address this limitation, we introduce an adaptive normalization strategy grounded in empirical quantiles.

Let $Q_{\alpha}$ and $Q_{\beta}$ denote the lower and upper empirical quantiles of the residual set $\{s_1,\ldots,s_N\}$, with $0 < \alpha < \beta < 1$.
We define winsorized residuals as
\begin{equation}
	\hat{s}_i =
	\min\left\{
	\max\left\{ s_i, \; Q_{\alpha} \right\},
	\; Q_{\beta}
	\right\},
	\label{eq:winsor}
\end{equation}
which bounds extreme values while preserving the relative ordering of residuals within the central portion of the distribution.

Normalized residuals are then computed according to
\begin{equation}
	z_i =
	\frac{\hat{s}_i - Q_{\alpha}}
	{Q_{\beta} - Q_{\alpha} + \varepsilon},
	\label{eq:znorm}
\end{equation}
where $\varepsilon > 0$ is a numerical regularization parameter introduced solely to avoid floating-point degeneracy.
The proposed adaptive indicator is defined by
\begin{equation}
	\mathcal{H}_{adap}(X)
	=
	-\frac{1}{N}
	\sum_{i=1}^{N}
	z_i \log \bigl( z_i + \varepsilon \bigr).
	\label{eq:Hadap}
\end{equation}

For theoretical analysis, we consider the idealized form of $\mathcal{H}_{adap}$ obtained by setting $\varepsilon = 0$.
This simplification preserves all structural properties of the indicator while facilitating exposition.

\begin{theorem}[Boundedness and scale invariance of the adaptive $\mathcal{H}$ indicator]
	Let $X=\{x_1,\dots,x_N\}$ be a finite approximation set and let $s_i=s(x_i)$ be the KKT stationarity residuals defined by
	\begin{equation}
		s(x)=\left\|\sum_{j=1}^m \lambda^\star_j(x)\nabla f_j(x)\right\|^2,
	\end{equation}
	where $\lambda^\star(x)$ solves the quadratic program in \eqref{eq:qpp}.
	Fix $0<\alpha<\beta<1$ and assume $Q_\beta>Q_\alpha$.
	Define $\mathcal{H}_{adap}(X)$ by \eqref{eq:Hadap} with $\varepsilon=0$.
	Then:
	\begin{enumerate}
		\item[(i)] (Boundedness) $0 \le \mathcal{H}_{adap}(X) \le \frac{1}{e}$.
		\item[(ii)] (Invariance under common positive scaling)
		If all objectives are scaled by a common factor $c>0$, i.e.,
		$f_j^{(c)}(x)=c\,f_j(x)$ for all $j$, then $\mathcal{H}_{adap}(X)$ is invariant under this transformation.
	\end{enumerate}
\end{theorem}

\begin{proof}
	(i) By construction, $\hat{s}_i \in [Q_\alpha,Q_\beta]$, which implies $z_i \in [0,1]$ for all $i$.
	Consider the function $\phi(t)=-t\log t$ on $[0,1]$ with $\phi(0)=0$.
	Elementary calculus shows that $\phi$ attains its maximum at $t=1/e$ and satisfies
	\[
	0 \le \phi(t) \le \frac{1}{e}
	\quad \text{for all } t \in [0,1].
	\]
	Therefore,
	\[
	0 \le \mathcal{H}_{adap}(X)
	= \frac{1}{N}\sum_{i=1}^N \phi(z_i)
	\le \frac{1}{N}\sum_{i=1}^N \frac{1}{e}
	= \frac{1}{e}.
	\]
	
	(ii) Let $f_j^{(c)}(x)=c f_j(x)$ with $c>0$.
	Then $\nabla f_j^{(c)}(x)=c\nabla f_j(x)$ and, for any feasible $\lambda$,
	\[
	\left\|\sum_{j=1}^m \lambda_j \nabla f_j^{(c)}(x)\right\|^2
	=
	c^2\left\|\sum_{j=1}^m \lambda_j \nabla f_j(x)\right\|^2.
	\]
	Hence, the optimizer $\lambda^\star(x)$ is unchanged and $s_i^{(c)}=c^2 s_i$.
	Since empirical quantiles are equivariant under positive scaling,
	$Q_\alpha^{(c)}=c^2 Q_\alpha$ and $Q_\beta^{(c)}=c^2 Q_\beta$.
	It follows that $z_i^{(c)}=z_i$ for all $i$, and therefore
	$\mathcal{H}_{adap}^{(c)}(X)=\mathcal{H}_{adap}(X)$.
\end{proof}

\begin{proposition}[Computational complexity of the adaptive $\mathcal{H}$ indicator]\label{prop_computacional}
	Let $N$ be the population size, $m$ the number of objectives, and $n$ the dimension of the decision space.
	The computational complexity of $\mathcal{H}_{adap}$ is $\mathcal{O}(N(nm^2 + m^3))$, with an additional overhead of $\mathcal{O}(N \log N)$ due to quantile estimation.
\end{proposition}

\begin{proof}
	The computation of $\mathcal{H}_{adap}(X)$ proceeds through distinct algorithmic phases.
	First, the stationarity residuals $\{s_1, \dots, s_N\}$ must be computed.
	For each solution $x_i$, this requires evaluating the Jacobian matrix $G \in \mathbb{R}^{n \times m}$ and solving the quadratic program in \eqref{eq:qpp}.
	The construction of the Hessian $H = 2G^T G \in \mathbb{R}^{m \times m}$ incurs a cost of $\mathcal{O}(nm^2)$.
	Solving the resulting convex quadratic program involves algebraic operations on an $m \times m$ matrix, leading to $\mathcal{O}(m^3)$ complexity.
	Repeating this process for the entire population yields a cost of $\mathcal{O}(N(nm^2 + m^3))$.
	
	Second, the computation of the empirical quantiles $Q_{\alpha}$ and $Q_{\beta}$ requires sorting the $N$ residuals, which has complexity $\mathcal{O}(N \log N)$.
	The subsequent winsorization, normalization, and entropy aggregation steps are linear in $N$ and therefore do not affect the asymptotic complexity.
	
	In many-objective regimes where $m$ is sufficiently large (e.g., $m \ge 30$), the constant factor associated with $m^3$ renders the per-solution cost of the quadratic programming phase substantially higher than the cost of sorting.
	As a result, the overhead introduced by adaptive quantile estimation is dominated by the computation of stationarity residuals in practical settings.
\end{proof}

\begin{remark}[Implications for ultra-many-objective optimization]
	In regimes characterized by a high number of objectives, the cubic dependence on $m$ in Proposition~\ref{prop_computacional} warrants consideration.
	For representative values such as $m=30$, the algebraic cost $m^3 = 27{,}000$ remains negligible for contemporary computing architectures, indicating that the quadratic programming phase does not constitute a computational bottleneck.
	
	By contrast, when analytical gradients are unavailable, numerical differentiation via finite differences requires $2m$ function evaluations per solution.
	For computationally expensive objectives, this may limit the practical applicability of the indicator to problems with low evaluation cost or readily available derivatives.
	
	Despite these considerations, the $\mathcal{H}_{adap}$ formulation is structurally advantageous in high-dimensional objective spaces.
	In ultra-many-objective scenarios, stationarity residuals commonly exhibit heavy-tailed distributions and high variance due to the curse of dimensionality.
	The fixed saturation mechanism of $\mathcal{H}_{old}$ collapses distinct non-stationary magnitudes into identical upper bounds, effectively eliminating discriminatory resolution.
	In contrast, the quantile-based normalization of $\mathcal{H}_{adap}$ adapts to the empirical distribution of residuals, preserving resolution among near-Pareto-stationary solutions and yielding more reliable convergence diagnostics for problems with $m \ge 30$.
\end{remark}

Both indicators satisfy $\mathcal{H}_{old}(X)\to 0$ and $\mathcal{H}_{adap}(X)\to 0$ as all points in $X$ approach Pareto stationarity.
Moreover, when the residual distribution lies entirely below the saturation threshold of $\mathcal{H}_{old}$, the adaptive formulation reduces to a monotone transformation of the original indicator.
This ensures conceptual consistency while providing increased robustness in heterogeneous convergence regimes.

In the following section, we present experimental results comparing $\mathcal{H}_{old}$ and $\mathcal{H}_{adap}$ with other classical performance indicators.

\section{Simulation Results}

This section investigates the empirical behavior of the proposed adaptive indicator
$\mathcal{H}_{adap}$ in comparison with the original formulation $\mathcal{H}_{old}$ and
with two classical performance metrics commonly employed in multi-objective optimization,
namely the averaged Hausdorff distance $\Delta_p$ and the Hypervolume (HV).
The experimental objective is not to assert the superiority of one indicator over another,
but rather to analyze consistency, resolution, and robustness across heterogeneous
many-objective scenarios.

The experimental evaluation is conducted on five benchmark problems from the DTLZ test
suite, namely DTLZ1 through DTLZ5.
These problems are widely adopted in the many-objective optimization literature due to
their ability to expose distinct convergence pathologies and geometric characteristics of
Pareto fronts.
DTLZ1 is characterized by a linear Pareto front combined with a large number of local optima,
making it suitable for evaluating convergence reliability.
DTLZ2 and DTLZ4 share a hyperspherical Pareto front but differ in the distribution of
objective scaling, with DTLZ4 introducing strong bias that challenges diversity preservation.
DTLZ3 extends DTLZ2 by incorporating a highly multimodal landscape, which exacerbates
dominance resistance and indicator degeneracy.
Finally, DTLZ5 induces a degenerate Pareto front with reduced intrinsic dimensionality,
providing a stringent test for the sensitivity of performance indicators to structural
irregularities.
Together, these problems form a representative benchmark set for assessing indicator
behavior under heterogeneous convergence regimes commonly encountered in many-objective
optimization.

All problems are instantiated with $m=12$ objectives in order to ensure a genuine
many-objective setting, in which dominance relations become sparse and indicator-based
assessment plays a central role.
This configuration reflects practical regimes where intrinsic convergence indicators are
most informative.

The algorithms considered in this study belong to a representative class of
state-of-the-art many-objective evolutionary optimization methods implemented in the
PlatEMO framework \cite{PlatEMO}.
They embody distinct algorithmic principles for addressing the challenges posed by
many-objective problems, including reference-point-based selection, cooperative constraint
handling, and geometry-driven search mechanisms.
In particular, the experimental evaluation includes:
\begin{itemize}
	\item NSGA-III, which extends nondominated sorting through a reference-point-based
	selection strategy tailored to many-objective optimization \cite{Deb2013};
	\item CMOEA-CD, a constrained many-objective evolutionary algorithm that exploits
	cooperative correlation mechanisms to balance feasibility and convergence
	\cite{CMOEACD2025};
	\item NRVMOEA, a normal-vector-guided evolutionary algorithm designed to enhance search
	efficiency and diversity maintenance in high-dimensional objective spaces
	\cite{Hua2024NRVMOEA}.
\end{itemize}

For all benchmark problems, a population size of $100$ individuals is employed.
Each algorithm is independently executed $30$ times, with a termination criterion of
$25{,}000$ objective function evaluations per run.
All simulations are carried out within the PlatEMO framework to ensure consistent
implementations and reproducibility of the reported performance metrics.

The performance comparison on the first benchmark problem is reported in
Table~\ref{tab:prob1}.
The table summarizes the mean and standard deviation of $\mathcal{H}_{adap}$,
$\mathcal{H}_{old}$, $\Delta_p$, and HV over the $30$ independent runs.
Lower values indicate better performance for $\mathcal{H}_{adap}$, $\mathcal{H}_{old}$, and
$\Delta_p$, while higher values are preferable for HV.

\begin{table*}[!htp]
	\centering
	\caption{Performance comparison on DTLZ1 ($m=12$).}
	\label{tab:prob1}
	\begin{tabular}{lcccc}
		\toprule
		Algorithm & $\mathcal{H}_{adap}\downarrow$ & $\mathcal{H}_{old}\downarrow$ & $\Delta_p\downarrow$ & HV$\uparrow$ \\
		\midrule
		NSGA-III  & $7.01\text{e-2}\;(8.19\text{e-2})$ & $3.50\text{e-2}\;(2.93\text{e-2})$ & $1.17\text{e+0}\;(1.94\text{e+0})$ & $5.56\text{e-1}\;(3.81\text{e-1})$ \\
		NRVMOEA  & $6.16\text{e-2}\;(8.88\text{e-2})$ & $3.46\text{e-2}\;(1.62\text{e-2})$ & $1.43\text{e+2}\;(1.59\text{e+1})$ & $0.00\text{e+0}\;(0.00\text{e+0})$ \\
		CMOEA-CD & $2.08\text{e-1}\;(9.88\text{e-3})$ & $2.32\text{e-3}\;(3.46\text{e-3})$ & $6.27\text{e+1}\;(5.07\text{e+1})$ & $3.80\text{e-3}\;(2.08\text{e-2})$ \\
		\bottomrule
	\end{tabular}
\end{table*}

The results in Table~\ref{tab:prob1} illustrate the loss of resolution exhibited by
classical indicators in high-dimensional objective spaces.
While $\mathcal{H}_{old}$ assigns similar values to NSGA-III and NRVMOEA, the adaptive
formulation $\mathcal{H}_{adap}$ yields clearer differentiation between the algorithms.
In contrast, the Hypervolume metric assigns values close to zero to NRVMOEA and CMOEA-CD,
indicating limited dominated volume coverage in this scenario.

Table~\ref{tab:prob2} reports the corresponding results for the second benchmark problem.

\begin{table*}[!htp]
	\centering
	\caption{Performance comparison on DTLZ2 ($m=12$).}
	\label{tab:prob2}
	\begin{tabular}{lcccc}
		\toprule
		Algorithm & $\mathcal{H}_{adap}\downarrow$ & $\mathcal{H}_{old}\downarrow$ & $\Delta_p\downarrow$ & HV$\uparrow$ \\
		\midrule
		NSGA-III  & $5.81\text{e-2}\;(3.76\text{e-2})$ & $2.28\text{e-2}\;(2.02\text{e-2})$ & $6.34\text{e-1}\;(5.43\text{e-2})$ & $9.14\text{e-1}\;(1.55\text{e-1})$ \\
		NRVMOEA  & $6.39\text{e-2}\;(1.60\text{e-2})$ & $1.33\text{e-2}\;(3.16\text{e-3})$ & $5.74\text{e-1}\;(5.27\text{e-3})$ & $9.67\text{e-1}\;(1.88\text{e-3})$ \\
		CMOEA-CD & $1.36\text{e-1}\;(1.43\text{e-2})$ & $1.98\text{e-1}\;(2.19\text{e-2})$ & $8.64\text{e-1}\;(1.25\text{e-1})$ & $4.57\text{e-1}\;(1.12\text{e-1})$ \\
		\bottomrule
	\end{tabular}
\end{table*}

For this problem, all indicators exhibit more regular behavior.
Nevertheless, the adaptive indicator introduces stronger differentiation between the
algorithms than the original formulation, while the Hypervolume metric confirms the superior
coverage achieved by NRVMOEA.

The results for the third benchmark problem are summarized in Table~\ref{tab:prob3}.

\begin{table*}[!htp]
	\centering
	\caption{Performance comparison on DTLZ3 ($m=12$).}
	\label{tab:prob3}
	\begin{tabular}{lcccc}
		\toprule
		Algorithm & $\mathcal{H}_{adap}\downarrow$ & $\mathcal{H}_{old}\downarrow$ & $\Delta_p\downarrow$ & HV$\uparrow$ \\
		\midrule
		NSGA-III  & $4.95\text{e-2}\;(2.57\text{e-2})$ & $1.46\text{e-1}\;(5.20\text{e-2})$ & $5.18\text{e+1}\;(2.72\text{e+1})$ & $0.00\text{e+0}\;(0.00\text{e+0})$ \\
		NRVMOEA  & $9.92\text{e-2}\;(1.84\text{e-2})$ & $2.45\text{e-1}\;(2.40\text{e-2})$ & $6.41\text{e+2}\;(4.51\text{e+1})$ & $0.00\text{e+0}\;(0.00\text{e+0})$ \\
		CMOEA-CD & $1.29\text{e-1}\;(2.96\text{e-2})$ & $2.33\text{e-1}\;(3.15\text{e-2})$ & $3.00\text{e+2}\;(1.46\text{e+2})$ & $0.00\text{e+0}\;(0.00\text{e+0})$ \\
		\bottomrule
	\end{tabular}
\end{table*}

In this setting, the Hypervolume metric becomes uninformative, as all algorithms attain zero
values.
While $\mathcal{H}_{old}$ exhibits compressed values due to saturation, the adaptive
formulation maintains a non-saturated scale across all algorithms, thereby preserving
effective resolution.

Tables~\ref{tab:prob4} and~\ref{tab:prob5} report the remaining results.

\begin{table*}[!htp]
	\centering
	\caption{Performance comparison on DTLZ4 ($m=12$).}
	\label{tab:prob4}
	\begin{tabular}{lcccc}
		\toprule
		Algorithm & $\mathcal{H}_{adap}\downarrow$ & $\mathcal{H}_{old}\downarrow$ & $\Delta_p\downarrow$ & HV$\uparrow$ \\
		\midrule
		NSGA-III  & $1.95\text{e-1}\;(1.35\text{e-2})$ & $8.37\text{e-8}\;(1.00\text{e-8})$ & $6.32\text{e-1}\;(3.01\text{e-2})$ & $9.46\text{e-1}\;(2.95\text{e-2})$ \\
		NRVMOEA  & $1.88\text{e-1}\;(1.26\text{e-2})$ & $9.37\text{e-8}\;(7.08\text{e-9})$ & $5.72\text{e-1}\;(2.70\text{e-3})$ & $9.72\text{e-1}\;(3.87\text{e-4})$ \\
		CMOEA-CD & $2.19\text{e-1}\;(9.75\text{e-3})$ & $1.80\text{e-7}\;(1.23\text{e-8})$ & $6.96\text{e-1}\;(1.41\text{e-2})$ & $7.14\text{e-1}\;(4.22\text{e-2})$ \\
		\bottomrule
	\end{tabular}
\end{table*}

Here, the original indicator collapses due to saturation effects, whereas the adaptive
formulation preserves meaningful differences between the algorithms and remains consistent
with the ordering suggested by Hypervolume.

\begin{table*}[!htp]
	\centering
	\caption{Performance comparison on DTLZ5 ($m=12$).}
	\label{tab:prob5}
	\begin{tabular}{lcccc}
		\toprule
		Algorithm & $\mathcal{H}_{adap}\downarrow$ & $\mathcal{H}_{old}\downarrow$ & $\Delta_p\downarrow$ & HV$\uparrow$ \\
		\midrule
		NSGA-III  & $1.05\text{e-1}\;(2.76\text{e-2})$ & $2.87\text{e-2}\;(1.25\text{e-2})$ & $8.00\text{e-1}\;(7.57\text{e-2})$ & $7.33\text{e-2}\;(1.18\text{e-2})$ \\
		NRVMOEA  & $9.26\text{e-2}\;(1.74\text{e-2})$ & $3.36\text{e-2}\;(1.05\text{e-2})$ & $1.36\text{e+0}\;(7.90\text{e-2})$ & $2.65\text{e-2}\;(2.45\text{e-2})$ \\
		CMOEA-CD & $8.97\text{e-2}\;(1.42\text{e-2})$ & $4.22\text{e-2}\;(6.46\text{e-3})$ & $2.62\text{e+0}\;(4.95\text{e-2})$ & $1.91\text{e-2}\;(2.61\text{e-2})$ \\
		\bottomrule
	\end{tabular}
\end{table*}

Overall, the experimental results indicate that $\mathcal{H}_{adap}$ behaves consistently
with established performance indicators while mitigating saturation and resolution loss
induced by fixed normalization schemes.
Rather than replacing classical metrics, the adaptive formulation serves as a complementary
tool for convergence assessment in many-objective optimization, particularly in scenarios
characterized by heterogeneous stationarity residual distributions.

\section{Conclusion}

This paper revisited an entropy-inspired KKT-based convergence indicator and proposed a
robust adaptive reformulation based on quantile normalization.
The proposed indicator preserves the stationarity-based interpretation of the original
formulation while improving robustness to heterogeneous distributions of stationarity
residuals, a recurring challenge in many-objective optimization.
By replacing fixed saturation thresholds with a distribution-aware normalization scheme,
the adaptive formulation enhances discriminatory power without introducing
problem-dependent tuning parameters or additional algorithmic complexity.

Empirical results obtained on a diverse set of DTLZ benchmark problems indicate that the
adaptive indicator behaves consistently with established convergence measures while
mitigating resolution loss and saturation effects observed in classical formulations.
In particular, scenarios in which the original indicator collapses to near-identical values
or reference-based metrics such as hypervolume become uninformative highlight the practical
relevance of adaptive normalization for convergence assessment in high-dimensional objective
spaces.

From a broader perspective, the proposed indicator is aligned with recent developments in
optimality-based proximity measures and intrinsic performance assessment in
multi-objective optimization.
Rather than competing directly with reference-based indicators or distance-based measures,
the adaptive formulation provides a complementary perspective grounded in first-order
optimality conditions.
This viewpoint is especially relevant in many-objective regimes, where dominance relations
weaken, reference information becomes unreliable, and indicator degeneracy is frequently
observed.

Future work may investigate theoretical convergence properties of the adaptive indicator,
its extension to constrained, noisy, and stochastic optimization settings, and its
integration into algorithmic components such as adaptive stopping criteria and online
convergence monitoring.
Such directions may further consolidate the role of optimality-driven indicators as
practical and reliable tools for convergence assessment in complex multi-objective
optimization problems.

\bibliographystyle{ieeetr}
\bibliography{bibliography}

\end{document}